\documentclass[10pt]{amsart}
\usepackage{amsmath,amssymb, amsbsy}
\usepackage[dvips]{graphicx}
\usepackage{color}
\usepackage{cancel}
\usepackage[latin1]{inputenc}
\usepackage[active]{srcltx}
\usepackage{esint}
\usepackage[usenames,dvipsnames]{pstricks}
\usepackage{epsfig}
\usepackage{pst-grad} % For gradients
\usepackage{pst-plot} % For axes
\usepackage[space]{grffile} % For spaces in paths
 \usepackage{etoolbox} % For spaces in paths
\makeatletter % For spaces in paths
\patchcmd\Gread@eps{\@inputcheck#1 }{\@inputcheck"#1"\relax}{}{}
\makeatother
\newcommand{\R}{{\mathbb R}}
\newcommand{\ren}{\R^{N}}

\newcommand{\N}{{\mathbb N}}

\newcommand{\bra}{\langle}
\newcommand{\ket}{\rangle}

\newcommand{\inn}{\text{  in   }}
\newcommand{\dyle}{\displaystyle}
\newcommand{\dint}{\dyle\int}

\renewcommand{\ge }{\geqslant}
\renewcommand{\geq }{\geqslant}
\renewcommand{\le }{\leqslant}
\renewcommand{\leq }{\leqslant}

\newtheorem{Theorem}{Theorem}[section]
\newtheorem{Corollary}[Theorem]{Corollary}
\newtheorem{Lemma}[Theorem]{Lemma}
\newtheorem{Proposition}[Theorem]{Proposition}
\theoremstyle{definition}
\newtheorem{Definition}[Theorem]{Definition}

\newtheorem{remark}[Theorem]{Remark}
\begin{document}
\title[Principal Eigenvalue of Mixed Fractional Problem]{Principal Eigenvalue of Mixed Problem for the Fractional Laplacian: Moving the Boundary Conditions}

\author[T. Leonori]{Tommaso Leonori}
\author[M. Medina]{Maria Medina}
\author[I. Peral]{Ireneo Peral}
\author[A. Primo]{Ana Primo}
\author[F. Soria]{Fernando Soria}

\address[Tommaso Leonori]{Departamento de An\'alisis Matematico,
Universidad de Granada, Granada, Espa\~na}
\email{{\tt leonori@ugr.es}}

\address[Maria Medina]{Facultad de Matem\'aticas, Pontificia Universidad Cat\'olica de Chile, Avenida Vicu\~na Mackenna 4860, Santiago, Chile}
\email{mamedinad@mat.uc.cl}

\address[Ireneo Peral, Ana Primo, Fernando Soria]{ Departamento de  Matem{\'a}ticas,  Universidad  Autónoma de  Madrid, 28049 Madrid, Espa\~na}
\email{{\tt ireneo.peral@uam.es}, {\tt ana.primo@uam.es}, {\tt fernando.soria@uam.es}}

\date{\today}

%\subjclass {35J20, 35J25, 31B10, 60J75, 35P20 }

\thanks{2010 {\it Mathematics Subject Classification AMS 2010. 35J20, 35J25, 31B10, 60J75, 35P20.} }
\thanks{Work partially supported by research grants MTM2013-40846-P and MTM2016-80474-P, MINECO, Spain. The second author was supported by a grant of FONDECYT-Postdoctorado, No. 3160077.}

 \begin{abstract}
We analyze the behavior of the eigenvalues of the following non local mixed problem
\begin{equation*}
\left\{
\begin{array}{rcll}
(-\Delta)^{s} u &=& \lambda_1(D) \ u &\inn\Omega,\\
u&=&0&\inn D,\\
\mathcal{N}_{s}u&=&0&\inn N.
\end{array}\right.
\end{equation*}
Our goal  is to construct different sequences of problems by modifying the configuration of the sets $D$ and $N$, and to provide sufficient and necessary conditions on the size and the location of these sets in order to obtain sequences of eigenvalues that in the limit recover the eigenvalues of the Dirichlet or Neumann problem. We will see that the non locality plays a crucial role here, since the sets $D$ and $N$ can have infinite measure, a phenomenon that does not appear in the local case (see for example \cite{D,D2,CP}).
 \end{abstract}

 %%%%%%%%%%%%%%%%%%%%%%%
%%%%%%%%%%%%%%%%%%%%%%%
%%%%%%%%%%%%%%%%%%%%%%%

\maketitle
\section{Introduction}
In the papers \cite{D, D2}, J. Denzler considers the following mixed Dirichlet-Neumann eigenvalue problem
\begin{equation}\label{local}
\left\{
\begin{array}{rcll}
-\Delta u&=&\lambda_1(D) u &\inn \Omega,\\[1.2 ex]
u&=&0 &\inn D,\\[0.4 ex]
\dfrac{\partial u}{\partial n}&=&0 &\inn N.
\end{array}
\right.
\end{equation}
Here $\Omega$ is a Lipschitz bounded domain in $\mathbb{R}^N,$  $D$, $N$ are submanifolds of $\partial \Omega$ such that
\begin{equation}\label{bcond}
\bar{D}\cup\bar{N}=\partial \Omega\;\;\mbox{ and }\;\;D\cap N=\emptyset,
\end{equation}
and $\lambda_1(D)$ is the first eigenvalue; that is, if $\displaystyle H^1_D(\Omega)=\{ u\in H^1(\Omega)\,|\, u=0 \hbox{  on  } D\subset\partial\Omega\}$, then
$$\lambda_1(D):=\inf_{u\in H_D^1(\Omega),\,\, u\not\equiv 0}  \frac{\dint_\Omega|\nabla u|^2\,dx}{\dint_\Omega | u|^2\,dx}\,.$$

In his paper, he studies the behavior of this eigenvalue according to the configuration of the sets with Dirichlet (or conversely, with Neumann) condition. More precisely, he constructs different examples describing the way in which the geometric arrangement  of the Dirichlet part (for a fixed measure) affects the size of the corresponding eigenvalue. Indeed, he shows the following property.

\begin{Theorem}[\cite{D2},  Theorem 5 and Theorem 6]\label{denzlerInf}
Given $0<\alpha\leq |\partial\Omega|$, and
$$\mu:=\inf\{\lambda_1(D):\;|D|=\alpha\}>0,$$
there exists a configuration set $D_0\subset \partial\Omega$ with $|D_0|=\alpha$ such that
$$\lambda_1(D_0)=\mu.$$
\end{Theorem}
That is, for an admissible value $\alpha$ and all the possible configurations  in problem \eqref{local} of the boundary conditions with the Dirichlet part given by a set of measure equal to $\alpha$, the infimum of the corresponding eigenvalues is positive and can be attained. In other words, there exists a configuration whose associated eigenvalue is this infimum.

We also have the following result
\begin{Theorem}[\cite{D}, Theorem 8]\label{denzlerSup}
For every $0<\alpha\leq |\partial\Omega|$,
$$\sup \{\lambda_1(D):|D|=\alpha\}=\lambda_1(\partial \Omega).$$
Moreover, a maximizing sequence  $D_n$ is given in such a way that the corresponding characteristic functions $\chi_n$
weakly converge  to a constant in $L^2(\partial\Omega)$.
\end{Theorem}
This theorem states that with the Dirichlet conditions, a tiny set as small as needed can be chosen in such a way that the eigenvalue problem behaves {\it almost} like the whole Dirichlet problem.
The physical interpretation of this property can be better explained by saying that for every $0<\alpha\le |\partial\Omega|$ we have a configuration in which the corresponding heat flux reflects an \textit{almost non isolated} situation.

Other kind of results concerning the configurations of the sets $D$ and $N$ along $\partial \Omega$ can be found in \cite{CP}. There, the authors consider sequences of sets with Dirichlet condition, $\{D_k\}_{k\in\N}$, and with Neumann condition, $\{N_k\}_{k\in\N}$, satisfying \eqref{bcond} for every $k$. They prove that if the sets are nested and their measure tends to zero, in the limit we recover the eigenvalue of the problem with the other condition in the whole boundary of $\Omega$. Namely, if
$$|N_k|\rightarrow 0\mbox{ when }k\rightarrow\infty\;\;\mbox{ and }\;\;N_{k+1}\subseteq N_k,$$
then $\lambda_1(D_k)\rightarrow \lambda_1(\partial\Omega)$.
Conversely, if
$$|D_k|\rightarrow 0\mbox{ when }k\rightarrow\infty\;\;\mbox{ and }\;\;D_{k+1}\subseteq D_k,$$
then $\lambda_1(D_k)\rightarrow 0$.
\medskip

In the spirit of these local results, the aim of this paper is to obtain necessary and sufficient conditions in order to construct sequences of boundary data to approximate the Dirichlet or Neumann eigenvalue in the non local setting. In the process we study how the size and position of the sets determine the eigenvalue of the associated mixed problem.

More precisely, we deal with the fractional Laplacian  $(-\Delta)^{s}$, with $0<s<1$, that is defined on smooth functions as follows:
\begin{equation}\label{frac}
(-\Delta)^{s}u(x)= a_{N,s} \dint_{\ren} \frac{u(x)-u(y)}{|x-y|^{N+2s}}\,dy,
\end{equation}
where the integral above has to be understood in the principal value sense and
$$
a_{N,s}=\left(\int_{\mathbb{R}^{N}}{\dfrac{1-cos(\xi_1)}{|\xi|^{N+2s}}d\xi}\right)^{-1}=2^{2s-1}\pi^{-\frac N2}\frac{\Gamma(\frac{N+2s}{2})}{|\Gamma(-s)|}\,,
$$
where $\Gamma$ denotes the Gamma Function.
As before, we are interested in the eigenvalue problem associated to such operator in a  bounded   domain.  Here and in all the paper (unless  explicitly specified) we will consider that $\Omega$ is a $C^{1,1}$ domain on which we assume the exterior sphere condition. To fix it once for all, we give the following definition.

\begin{Definition}
We say that $\Omega$ is an \textit {admissible} domain in this context if it is a $C^{1,1}$ domain with the exterior sphere condition.
\end{Definition}

As far as the boundary conditions are concerned, we suppose that the complementary of $\Omega$ is divided into two sets in which we prescribe the Dirichlet and the Neumann conditions.
In fact, we suppose that  $N$ and $D$ are two open sets of positive measure satisfying
\begin{equation} \label{DN}
D,N \subset \Omega^c, \quad D\cap N=\emptyset, \ \ \ \ \left | \mathbb R^N\setminus (\Omega\cup D \cup N)\right |=0   \,.
\end{equation}

Thus we consider the following problem:
\begin{equation}\label{prob}
\left\{
\begin{array}{rcll}
(-\Delta)^{s} u &=& \lambda_1(D) \ u &\inn\Omega,\\
u&=&0&\inn D,\\
\mathcal{N}_{s}u&=&0&\inn N,
\end{array}\right.
\end{equation}
where  $\mathcal{N}_{s}$ denotes  the \textit{nonlocal normal derivative}.

Several definitions  of {non local normal derivative} can be found in the literature. We   use the one proposed by S. Dipierro, X. Ros-Oton and E. Valdinoci in \cite{DRV}, given on smooth functions $u$  by
\begin{equation}\label{normal}
\mathcal{N}_{s}u(x):=a_{N,s} \dint_{\Omega} \dfrac{u(x)-u(y)}{|x-y|^{N+2s}}\,dy, \qquad  x\in \Omega^c .
\end{equation}
The use of this Neumann boundary condition is  justified, among others,  by two reasons:
\begin{enumerate}
\item[(i)] A Gauss-type formula holds (see \eqref{gauss} below);

\item[(ii)] The problem admits a variational formulation   (see \eqref{hilbert}).

\end{enumerate}
As far as the first eigenvalue is concerned, we define $\lambda_1 (D) $ as
$$
\lambda_1 (D) =\inf\limits_{u\in H^s_D,\,\|u\|_{L^2(\Omega)}=1} \iint_{\R^{2N} \setminus (\Omega^c)^2 } \dfrac{|u(x)-u(y)|^2}{|x-y|^{N+2s}}\,dxdy,
$$
where $H^s_D(\Omega):=   \{ u \in H^s (\Omega)  \mbox{ such that }\   u=0\inn D  \}$.

\medskip

Let us observe that  the study of the different configurations of $D$ and $N$ in \eqref{prob} is much more involved in the fractional case than in the local one, since the role played by the  {\it boundary} of $\Omega$ in now replaced by  the whole $\Omega^c$, and both sets $D$ and $N$ may change  in many different ways.
 Indeed we have to take care not only of the size of the sets $N$ and $D$ (that are allowed to be, one or both of them, of infinite Lebesgue measure)  but also of the shape and, in some sense, of how far they are located with respect to  $\Omega$.

 \medskip

 The first  result we prove deals with the characterization of how to arrange  a sequence of domains $\{N_k\}_{k\in \mathbb{N}}$, in which the Neumann condition is prescribed, in order to prove that the corresponding (first) eigenvalue gets close to the one with Dirichlet condition in the whole of $\Omega^c$.

  \medskip

 Our main result in this direction is the following.

 \begin{Theorem}\label{theoNeu}
Let $\Omega$ be admissible  and consider  $D_k,\, N_k\subset \Omega^c$  open sets such  that
\begin{equation}\label{nkdk}
D_k,N_k \subset \Omega^c, \quad D_k\cap N_k=\emptyset, \ \ \ \ \left | \mathbb R^N\setminus (\Omega\cup D_k \cup N_k)\right |=0  .
\end{equation}
Then,  the following two statements are equivalentes:
\begin{enumerate}
\item [(A)] $\displaystyle\lim\limits_{k\rightarrow \infty}\lambda_{1} (D_k) =\lambda_{1} (\Omega^c) $;\\[1mm]
\item  [(B)] The Neumann sets $\{N_{k}\}_{k\in \N}$ diffuse to zero on compact sets; that is, $\forall R>0$,
\\[2mm]
$ \displaystyle \lim\limits_{k\rightarrow\infty} |N_{k}\cap B_{R}|=0,$ where  $B_R$ denotes the ball of radius $R$ centered at the origin.
\end{enumerate}
\end{Theorem}

 \medskip

Let us observe that here we have a different scenario with respect to the local setting. Indeed we can find several configurations of the sets $N_k$ and $D_k$ for which the above convergence holds true (see Section \ref{2} for more examples). Independently of the measures of the two  sets, we have convergence of the first eigenvalue with mixed boundary condition if, for instance, the Neumann sets \lq\lq travel" away to infinity. In other words, we have a sort of equivalence of the result obtained in \cite{CP} without requiring that the sets get nested.

 \medskip

Actually, for certain range of $s$, we have a similar result that assures  the convergence of the sequence $\{\lambda_1 (D_k)\}_k$ to zero.

\begin{Theorem}\label{introd}
Let $0<s<1/2$, $\Omega$  a domain and let $D_k,\, N_k\subset \Omega^c$ be  as in
\eqref{nkdk}. Then,
$$ \hbox{ for any }R>0,
\; \lim\limits_{k\rightarrow\infty} |D_{k}\cap B_{R}|=0 \,
\quad \Longleftrightarrow \quad \lim\limits_{k\rightarrow \infty}\lambda_{1} (D_k) =0\, ,
\;.$$
\end{Theorem}

(For  more precise statements on this problem, see Theorems \ref{N1} and \ref{N2} in Section \ref{secDir}.)

\medskip

Even in this case the situation is quite different from the one treated by J. Denzler in \cite{D,D2}:  indeed, given a fixed positive value, we can find a configuration where the measure of the Dirichlet set has that particular value and  such  that the associated eigenvalue is as small as we want, in clear contrast to what happens in the local case (see Theorem  \ref{denzlerInf}).

\medskip

\medskip

The paper is organized as follows. In Section \ref{1} we show the functional framework associated to the non local problem \eqref{prob} and the Neumann boundary condition, as well as some integrability results related to the geometry of $\Omega$ that will be fundamental in the study of the convergence of the sequences of eigenvalues. Section \ref{2} is devoted to analyze the non local mixed eigenvalue problem. We describe the main properties of the associated eigenvalues, and we provide examples of different possible configurations for the sets $D$ and $N$. These examples motivate the results obtained in the following sections. In Section \ref{secNeu} we consider the case when the sets with Neumann boundary condition {\it tend to disappear}, that is, we study sequences of problems where the measure of the Neumann part tends to zero, or the set goes far from $\Omega$. Finally, in Section \ref{secDir} we perform the opposite analysis, that is, when the Dirichlet part {\it decreases}.

%%%%%%%%%%%%%%%%%%%%%%%
%%%%%%%%%%%%%%%%%%%%%%%
%%%%%%%%%%%%%%%%%%%%%%%

\section{Functional framework and preliminary results.}\label{1}
\setcounter{equation}{0}

Consider the general non local elliptic mixed problem,
\begin{equation}\label{probl}
\left\{
\begin{array}{rcl}
(-\Delta)^{s} u &=& f \;\;\inn\Omega,\\
u&=&0\;\;\inn D,\\
\mathcal{N}_{s}u&=&0\;\;\inn N,
\end{array}\right.
\end{equation}
where $\Omega$ is a bounded Lipschitz domain of $\ren$, $s\in (0,1)$ and $\mathcal{N}_s$ is the Neumann condition defined in \eqref{normal}. Here $N$ and $D$ are two open sets of positive measure satisfying
$$
D\cap N=\emptyset, \ \ \ \ \bar{D}\cup\bar{N}=\Omega^c,
$$
and $f\in \mathcal{C}_{0}^{\infty}(\Omega)$, $f\geq 0$.

Let $u,v: \mathbb{R}^{N}\rightarrow\mathbb{R}$ be measurable functions and denote $Q_\Omega:= \mathbb{R}^{2N}\setminus (\Omega^c)^{2}$. Consider the scalar product
\begin{equation}\label{scalar}
\bra u, v\ket_{H^{s}_{D}(\Omega)}:= \dint_{\Omega} u v\,dx+\iint_{Q_\Omega} \dfrac{ (u(x)-u(y))(v(x)-v(y))}{|x-y|^{N+2s}}\,dx dy,
\end{equation}
and the associated norm
$$
\|u\|^{2}_{H^{s}_{D}(\Omega)}:= \dint_{\Omega} u^2\,dx+\iint_{Q_\Omega} \dfrac{ |u(x)-u(y)|^2}{|x-y|^{N+2s}}\,dx dy.
$$
Thus, we define the space
\begin{equation}\label{hilbert}
H^{s}_{D}(\Omega) := \{ u: \mathbb{R}^{N}\rightarrow \mathbb{R} \mbox{ measurable, such that } \|u\|_{H^{s}_{D}(\Omega)}<\infty  \mbox { and } u=0\inn D\},
 \end{equation}
that is a Hilbert space with the scalar product defined in \eqref{scalar}.

As we pointed out before, an advantage of the definition \eqref{normal} is that problem \eqref{probl} has a variational structure. In particular, for $u$ and $v$ bounded $C^2$ functions in $\mathbb{R}^{N}$, the classical integration by parts formulae for the Laplacian operator,
$$
\dint_{\Omega} \Delta u \,dx=\dint_{\partial\Omega} \partial_{\nu} u\, d\sigma
\qquad \mbox{ and }
\qquad
\dint_{\Omega} \nabla u\nabla v\,dx= \dint_{\Omega} v (-\Delta) u\,dx+\dint_{\partial\Omega} v \ \partial_{\nu}u\,d\sigma,
$$
are replaced by
\begin{equation}\label{gauss}
\dint_{\Omega} (-\Delta)^{s} u\,dx=-\dint_{\mathbb{R}^{N}\setminus\Omega} \mathcal{N}_{s} u\,dx
\end{equation}
and
\begin{equation}\label{parts}
\dfrac{a_{N,s}}{2} \iint_{Q_\Omega} \dfrac{(u(x)-u(y))(v(x)-v(y))}{|x-y|^{N+2s}}dxdy=\dint_{\Omega} v (-\Delta)^{s}u\, dx+\dint_{\mathbb{R}^{N}\setminus\Omega} v \mathcal{N}_{s}u\, dx
\end{equation}
(see Lemma 3.2 and Lemma 3.3 in \cite{DRV}). Actually, these identities motivate the notion of solution  we  use, which is the following.
\begin{Definition}
Let $f\in L^{2}(\Omega)$. We say that $u\in H_{D}^{s}(\Omega)$ is a weak solution of \eqref{probl} if
$$
\dfrac{a_{N,s}}{2}\iint_{Q_\Omega} \dfrac{(u(x)-u(y))(\varphi(x)-\varphi(y))}{|x-y|^{N+2s}}\, dxdy=\dint_{\Omega} f\varphi\,dx,
$$
for every $\varphi\in H_{D}^{s}(\Omega)$.
\end{Definition}
Furthermore, taking advantage of the fact of having Dirichlet boundary condition on some region, one can establish a Poincar\'e type inequality (see Proposition 2.4 in \cite{BM}).
\begin{Proposition}\label{Poin} (Poincar\'e inequality) There exists a constant $C=C(\Omega, N,s)>0$ such that
$$
\dint_{\Omega} u^2\,dx\leq C \iint_{Q_\Omega} \dfrac{|u(x)-u(y)|^2}{|x-y|^{N+2s}}\,dxdy,
$$
for every $u\in H_{D}^{s}(\Omega)$.
\end{Proposition}
As a consequence of this result the coercivity of the operator in $H_{D}^{s}(\Omega)$ holds and Lax-Milgram Theorem can be applied to guarantee the existence and uniqueness of a weak solution to \eqref{probl} when $f\in L^{2}(\Omega)$.

To end this section, we state two summability results that we will need in the sequel.
\begin{Lemma}\label{punctual}
Fix $\alpha>0$. Given $\Omega\subset\mathbb{R}^{N}$ and $x \in \Omega^c$, we have
$$
I_{\Omega}^{\alpha} (x) =\dint_{\Omega} \dfrac{dy}{|x-y|^{N+\alpha}}\leq C \dfrac{1}{\mbox{dist} (x,\partial\Omega)^{\alpha}},
$$
where $C$ is a positive constant depending only on $\alpha, \Omega$ and $N$.
\end{Lemma}
\begin{proof} Set $r:=\mbox{dist}(x,\partial\Omega)$ and define the ball $B:=B_{r}(x)$. For $y\in\Omega$, we have $|x-y|\geq r$ and thus $y\notin\dot{B}$ and
$$
I_{\Omega}^{\alpha}(x)\leq \dint_{\mathbb{R}^{N}\setminus B} \dfrac{dy}{|x-y|^{N+\alpha}} = w_{N-1}\dint_{r}^{\infty} \dfrac{dt}{t^{1+\alpha}}= \frac{w_{N-1}}{\alpha}\dfrac{1}{r^{\alpha}}.
$$
\end{proof}
\begin{remark}
The estimate above is sharp in the following sense:

\noindent Assume that for some $k>1$, there exists $\delta>0$, such that one has $|B_{kr}(x)\cap \Omega| \geq \delta |B_{kr}(x)|$. Then,
$$
I_{\Omega}^{\alpha}(x)\geq \dint_{B_{kr}(x)\cap\ \Omega}\dfrac{dy}{|x-y|^{N+\alpha}}\geq \dfrac{1}{(kr)^{N+\alpha}} |B_{kr}(x)\cap \Omega| \geq \delta |B_{kr}(x)|\geq \delta \dfrac{|B_{kr}(x)|}{(kr)^{N+\alpha}} = \dfrac{c}{r^{\alpha}}.
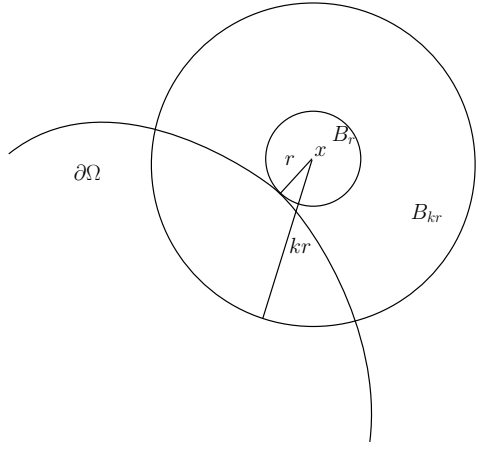
$$
\end{remark}

\begin{center}
\psscalebox{0.4 0.4}
%\psscalebox{1.0 1.0} % Change this value to rescale the drawing.
{
\begin{pspicture}(0,-7.311224)(15.532545,7.311224)
\pscircle[linecolor=black, linewidth=0.04, dimen=outer](10.132545,1.9112242){5.4}
\pscircle[linecolor=black, linewidth=0.04, dimen=outer](10.132545,2.1112242){1.6}
\psbezier[linecolor=black, linewidth=0.04](0.012545014,2.2712243)(3.1578448,4.8019495)(7.572545,2.1045575)(8.772545,1.1712242126464845)(9.972545,0.23789088)(12.476333,-3.5546448)(12.012545,-7.308776)
\psline[linecolor=black, linewidth=0.04](10.092545,2.1112242)(9.032545,0.9512242)(9.052545,0.9312242)
\psline[linecolor=black, linewidth=0.04](10.092545,2.0712242)(8.452545,-3.2087758)
\rput[bl](2.172545,1.4112242){\Huge $  \partial \Omega$}
\rput[bl](9.192545,1.9){\Huge $ r$}
\rput[bl](9.332545,-1.0287758){\Huge$ k r$}
\rput[bl](10.752545,2.6){\Huge$ B_r$}
\rput[bl](13.372545,0.011224213){\Huge $ B_{kr}$	}
\rput[bl](10.172545,2.2112243){\Huge $x$}
\end{pspicture}
}
\end{center}
\begin{remark}\label{chi}
Notice that if $0<\alpha<1$ we have that  $I_{\Omega}^{\alpha}(x)= - (-\Delta)^{\frac{\alpha}{2}} \chi_{\Omega}(x)=-\dint_{\mathbb{R}^{N}} \, \dfrac{\chi_{\Omega}(x)-\chi_{\Omega}(y)}{|x-y|^{N+\alpha}}dy$, $x\in\Omega^c$. This connects the integrability of this function in $\mathbb R^N\setminus \Omega$ with the norm of $\chi_\Omega$ in the Sobolev space $H^{\alpha/2} ({\mathbb R^N})$. To be more precise, we have

\begin{align*}
\int_{\mathbb R^N\setminus \Omega} I_{\Omega}^{\alpha}(x)dx &=
\int_{\mathbb R^N}  I_{\Omega}^{\alpha}(x)\chi_{ \Omega^c}(x) dx =
\iint_{\mathbb R^N \times \mathbb R^N} \frac{\chi_\Omega(y) \, \chi_{\Omega^c}(x)}{|x-y|^{N+\alpha}}dydx \\[2mm]
&= \frac 12 \iint_{\mathbb R^N \times \mathbb R^N} \frac{\left(\chi_\Omega(x) - \chi_{\Omega}(y)\right)^2}{|x-y|^{N+\alpha}}dxdy
=C_{N,\alpha} \|\chi_\Omega\|_{H^{\alpha/2  } (\R^N) }^2.
\end{align*}
\end{remark}

Next lemma deals with the local summability of a negative power of the distance to $\partial \Omega$. As we will see, the result  is true for domains not necessarily $C^{1,1}$.

\begin{Lemma} \label{regularidad}
Fix $0<\alpha<1$ and consider $\Omega\subset\mathbb{R}^{N}$ a  bounded domain where $\partial\Omega$ is of class $\mathcal{C}^{\beta}$, with $\beta>\alpha$. If $x \in \Omega^c$, then $\dfrac{1}{(\mbox{dist} (x,\partial\Omega))^{\alpha}}\in L^{1}_{loc}(\mathbb{R}^{N}\setminus\Omega)$.
\end{Lemma}
\begin{proof}
Without loss of generality, we can assume that $\partial\Omega$ is (locally) the graph of a {positive} function $\phi$ defined on a ball $B\subset\mathbb{R}^{N-1}$, satisfying $|\phi(x')-\phi(y')|\leq C_{\phi} |x'-y'|^{\beta}$.
%That is,
%$$
%\partial\Omega=\{(x', \phi(x')): x'\in B\}.
%$$
Then, it suffices to prove that if $T>0$ is large enough and $D$ denotes the shaped cylinder
$$
D:=\{ (x',x_n): x'\in B, \ \phi(x')\leq x_n\leq T\},
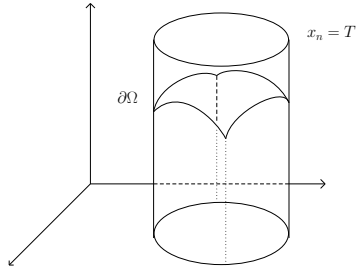
$$
\begin{center}
\psscalebox{0.3 0.3} % Change this value to rescale the drawing.
{
\begin{pspicture}(0,-5.817108)(14.358463,5.817108)
\psline[linecolor=black, linewidth=0.04](3.6484637,5.802966)(3.6484637,-2.197034)(3.6484637,-2.197034)(0.048463747,-5.797034)(0.048463747,-5.797034)
\psline[linecolor=black, linewidth=0.04](3.6484637,-2.197034)(6.448464,-2.197034)(6.448464,-2.197034)
\psellipse[linecolor=black, linewidth=0.04, dimen=outer](9.448463,-4.3970337)(3.0,1.4)
\psellipse[linecolor=black, linewidth=0.04, dimen=outer](9.448463,4.202966)(3.0,1.2)
\psline[linecolor=black, linewidth=0.04](6.448464,4.202966)(6.448464,4.202966)(6.448464,-4.597034)(6.448464,-4.597034)
\psline[linecolor=black, linewidth=0.04](12.448463,4.202966)(12.448463,-4.197034)(12.448463,-4.597034)(12.448463,-4.597034)
\psline[linecolor=black, linewidth=0.04, linestyle=dotted, dotsep=0.10583334cm](9.648464,-5.797034)(9.648464,-0.19703384)(9.648464,-0.19703384)
\psbezier[linecolor=black, linewidth=0.04](9.648464,-0.19703384)(9.248464,0.6029661)(7.6484637,2.2029662)(6.448464,1.0029661560058594)
\psbezier[linecolor=black, linewidth=0.04](12.448463,1.4029661)(12.248464,2.2029662)(9.848464,1.0029662)(9.648464,-0.19703384399414062)
\psbezier[linecolor=black, linewidth=0.04](9.248464,2.602966)(8.448463,3.0029662)(6.6484637,2.2029662)(6.448464,1.0029661560058594)
\psbezier[linecolor=black, linewidth=0.04](9.248464,2.602966)(9.648464,3.0029662)(11.648464,3.0029662)(12.448463,1.4029661560058593)
\psline[linecolor=black, linewidth=0.04](12.448463,-2.197034)(14.048464,-2.197034)(14.048464,-2.197034)
\rput[bl](4.8484635,1.4029661){\Huge $\partial \Omega$}
\rput[bl](13.248464,4.202966){\Huge $x_n=T$}
\psline[linecolor=black, linewidth=0.04, linestyle=dashed, dash=0.17638889cm 0.10583334cm](6.448464,-2.197034)(12.448463,-2.197034)(12.448463,-2.197034)
\psline[linecolor=black, linewidth=0.04, linestyle=dashed, dash=0.17638889cm 0.10583334cm](9.248464,0.6029661)(9.248464,2.602966)(9.248464,2.602966)
\psline[linecolor=black, linewidth=0.04, linestyle=dotted, dotsep=0.10583334cm](9.248464,0.20296615)(9.248464,-2.9970338)(9.248464,-2.9970338)
\psline[linecolor=black, linewidth=0.04](3.6484637,5.802966)(3.4484637,5.6029663)(3.4484637,5.6029663)
\psline[linecolor=black, linewidth=0.04](3.8484638,5.6029663)(3.6484637,5.802966)
\psline[linecolor=black, linewidth=0.04](0.048463747,-5.797034)(0.048463747,-5.597034)
\psline[linecolor=black, linewidth=0.04](0.048463747,-5.797034)(0.24846375,-5.797034)
\psline[linecolor=black, linewidth=0.04](14.048464,-2.197034)(13.848464,-1.9970338)
\psline[linecolor=black, linewidth=0.04](14.048464,-2.197034)(13.848464,-2.397034)
\end{pspicture}
}
\end{center}
\noindent we have $$\dint_{D} \dfrac{dx}{\mbox{dist}(x,\partial\Omega)^{\alpha}}<\infty.$$

\vskip 2mm

\noindent Let us start with a Lipschitz function $\phi$. To that end we make first the following

\vskip 2mm

\noindent {{\bf Claim 1:}} There exists $c_0>0$ such that  for all $x=(x',x_n)\in D$ one has
\begin{equation*}
\mbox{dist} (x,\partial\Omega)\geq c_0|x_n-\phi(x')|.
\end{equation*}

\vskip 2mm

\noindent Assuming that the claim is true, using that $0<\alpha<1$ we  have
\begin{equation*}\begin{split}
\dint_{D} \dfrac{1}{\mbox{dist} (x,\partial\Omega)^{\alpha}}\,dx&\leq \frac{1}{c_0^\alpha} \dint_{B} \Big( \dint_{\phi(x')}^{T} \dfrac{dx_n}{|x_n-\phi(x')|^{\alpha}}\Big)\, dx'\\
&= \frac{1}{c_0^\alpha} \dint_{B} \dfrac{1}{1-\alpha} \big( T-\phi(x')\big)^{1-\alpha} \leq C_{\alpha} T^{1-\alpha} |B|,
\end{split}\end{equation*}
and the lemma follows.

\vskip 2mm

\textit {Proof of   Claim 1:}
Let $x=(x', x_n)\in D$ and define $\bar{x}:=(x',\phi(x'))$. Take now $\bar{y}=(y',\phi(y'))$ with $y'\in B$.
We want to prove that  $ \exists c_0 >0$ such that $|x-\bar{y}| \geq c_0|x_n-\phi(x')|$, $\forall \overline{y}$.  \,Let us consider two cases:

\textit{Case 1}: $|x_n-\phi(y')|\ge \frac 12 |x_n-\phi(x')|$. In this case, since $|x-\bar{y}|\ge  |x_n-\phi(y')|$, the result is obvious.

\textit{Case 2}: $|x_n-\phi(y')|\le \frac 12 |x_n-\phi(x')|$. Since $\phi$ is Lipschitz, there exists a constant $C_\phi$ such that
$|\phi(x')-\phi(y')|\leq C_{\phi} |x'-y'|$. Hence,
$$
|x_n-\phi(x')|\le |x_n-\phi(y')|+|\phi(x')-\phi(y')|\le \frac 12 |x_n-\phi(x')|+C_{\phi} |x'-y'|.
$$
It means that $|x_n-\phi(x')|\le 2C_{\phi} |x'-y'|$, and since $|x-\bar{y}|\ge  |x'-y'|$, the result follows with $c_0=\frac 1{2C_\phi}$.

\vskip 2mm
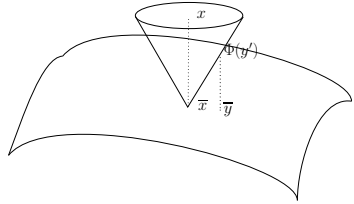
\begin{center}
\psscalebox{0.3 0.3} % Change this value to rescale the drawing.
{
\begin{pspicture}(0,-4.4)(15.238736,4.4)
\psbezier[linecolor=black, linewidth=0.04](2.418736,2.0)(4.418736,4.4)(15.218736,1.6)(15.218736,0.0)
\psbezier[linecolor=black, linewidth=0.04](0.018735962,-2.4)(2.018736,0.0)(12.818736,-2.8)(12.818736,-4.4)
\psbezier[linecolor=black, linewidth=0.04](2.418736,2.0)(1.4987359,1.96)(0.63873595,-0.74)(0.018735962,-2.4)
\psbezier[linecolor=black, linewidth=0.04](15.218736,0.0)(13.698736,0.24)(12.818736,-2.8)(12.818736,-4.4)
\psellipse[linecolor=black, linewidth=0.04, dimen=outer](8.018736,3.8)(2.4,0.6)
\psline[linecolor=black, linewidth=0.04](5.638736,3.76)(7.958736,-0.26)(10.4187355,3.74)
\rput[bl](8.4187355,-0.38){\Huge  $\overline{x}$}
\psline[linecolor=black, linewidth=0.04, linestyle=dotted, dotsep=0.10583334cm](7.978736,-0.24)(7.998736,3.76)(7.998736,3.76)
\rput[bl](8.358736,3.7){\Huge $x$}
\psline[linecolor=black, linewidth=0.05, linestyle=dotted, dotsep=0.10583334cm](9.398736,1.96)(9.398736,-0.42)(9.398736,-0.38647887)
\rput[bl](9.558736,1.88){\Huge $\Phi (y')$}
\rput[bl](9.558736,-0.64){\Huge $\overline{y}$}
\end{pspicture}
}
\end{center}

\vskip 2mm

Continuing with the proof of the lemma, for a function $\phi\in\mathcal {C}^{\beta}$, with $\beta>\alpha$, we obtain that if $\bar{y}=(y',\phi(y'))$, then $\phi(y')$ is below
$$
\phi(x')+C_{\phi} |x'-y'|^{\beta}.
$$
This is better reflected in the following claim, whose proof is similar to the previous one:

\vskip 2mm

\noindent {{\bf Claim 2}}: There exists $c_0>0$ such that  for all $x=(x',x_n)\in D$ one has
\begin{equation*}
\mbox{dist} (x,\partial\Omega)\geq c_0\min\left(|\phi(x')-x_n|, \, |\phi(x')-x_n|^{\frac 1\beta}\right).
\end{equation*}

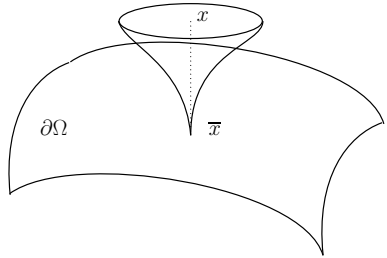
\begin{center}
\psscalebox{0.4 0.4} % Change this value to rescale the drawing.
{
\begin{pspicture}(0,-4.2028)(12.720013,4.2028)
\psbezier[linecolor=black, linewidth=0.04](0.26081267,-2.1372004)(2.3608127,-0.6372003)(10.000813,-1.8572003)(10.680813,-4.197200317382812)
\psbezier[linecolor=black, linewidth=0.04](2.2808127,2.2427998)(4.3808126,3.7427998)(12.020813,2.5227997)(12.700812,0.1827996826171875)
\psbezier[linecolor=black, linewidth=0.04](0.28081268,-2.1772003)(0.020812683,0.4027997)(1.0808127,1.8027997)(2.2408128,2.2227996826171874)
\psbezier[linecolor=black, linewidth=0.04](10.680813,-4.1772003)(10.420813,-1.5972003)(11.480813,-0.19720031)(12.640813,0.2227996826171875)
\psellipse[linecolor=black, linewidth=0.04, dimen=outer](6.280813,3.6027997)(2.4,0.6)
\psline[linecolor=black, linewidth=0.04, linestyle=dotted, dotsep=0.10583334cm](6.2808127,3.6027997)(6.2808127,-0.19720031)
\psbezier[linecolor=black, linewidth=0.04](6.2808127,-0.19720031)(6.2808127,2.6027997)(8.680813,2.8027997)(8.680813,3.6027996826171873)
\psbezier[linecolor=black, linewidth=0.04](3.8808126,3.6027997)(4.0808125,2.8027997)(6.0808125,2.2027998)(6.2808127,-0.1972003173828125)
\rput[bl](6.8808126,-0.19720031){\Huge $\overline{x}$}
\rput[bl](6.4808125,3.6027997){\Huge $x$}
\rput[bl](1.2808127,-0.19720031){\Huge $ \partial \Omega$}
\end{pspicture}
}
\end{center}

\noindent Now, since $\beta>\alpha$,
\begin{equation*}\begin{split}
\dint_{D} \dfrac{1}{\mbox{dist} (x,\partial\Omega)^{\alpha}}\,dx&\leq \frac{1}{c_0} \dint_{B}  \dint_{\phi(x')}^{T} \left( \dfrac{1}{|x_n-\phi(x')|^\alpha}+\dfrac{1}{|x_n-\phi(x')|^{\frac{\alpha}{\beta}}}\right)\,dx_n dx' \\
&\leq C_{\alpha,\beta, T}|B| <+\infty,
\end{split}\end{equation*}
and the lemma follows in this case, too.
\end{proof}

\begin{remark}
As one can readily see, we actually have the $L^p$ local summability  $(\mbox{dist} (x,\partial\Omega))^{-\alpha}\in L^{p}_{\rm loc}(\mathbb{R}^{N}\setminus\Omega)$,  $\forall p\in [1, \frac{\beta}\alpha)$.
\end{remark}

\medskip

\begin{remark}\label{geometric}
There is a simple geometric proof of  Claim 1 that we want to outline: since $\phi$ is Lipschitz, given $x',y'\in B$, we have that $\phi(y')$ is \lq\lq below"\,  the cone $\phi(x')+C_{\phi} |x'-y'|$. To be more precise, for every $\bar{x}=(x',\phi(x'))$ there exists a cone with vertex at $\bar x$ of exclusion for the other points in $\partial \Omega$ and  determined by
$$
\Gamma(\bar{x}):=\{z=(z',z_n): z_n>\phi(x')+C_{\phi} |x'-z'|\}.
$$
That is, for $\bar{y}=(y',\phi(y'))$, $y'\neq x'$, we have that $\bar y\notin L(\bar x)$ and hence
$$
\mbox{dist} (x,\partial\Omega)\geq \mbox{dist} (x,\partial \Gamma(\bar x)) \sim |x_n-\phi(x')|.
$$
\end{remark}

\medskip

\subsection{A digression on local integrability of the fractional Laplacian}

One may find some similarities between the above estimates and the theory of regularity for nonlocal minimal surfaces developed by Caffarelli-Roquejoffre-Savin \cite{CRS} (see also  the work of Barrios-Figalli-Valdinoci \cite{BFV}). In this theory one fixes the difussion $s\in (0,1)$ determined by the Euler-Lagrange equation for a set $E$
$$
\int_{\mathbb R^N} \frac{\chi_E(x)-\chi_E(y)}{|x-y|^{N+s}}dy=0, \quad x\in \partial E,
$$
and then finds, accordingly, the regularity of $\partial E$. This is done by looking for minimizers that live in the Sobolev space $H^{s/2}$. Our approach here goes somehow in the opposite direction. We consider regularity on the boundary and then determine the Sobolev space to which the characteristic function of the set belongs. Thus, from Lemmas  \ref{punctual}, \ref{regularidad} and Remark \ref{chi} one can roughly conclude that if $\partial \Omega$ has  smoothness $\beta>\alpha$ then $\chi_\Omega$ is in $H^{\alpha/2}$. We want to give a more general result along these lines.  To that end we make the following definition.

\begin{Definition}
We say that the function $\omega_0: [0,\infty) \to [0,\infty) $ is a modulus of continuity if it is continuous, increasing and $\omega_0(0)=0$.
\end{Definition}
The next result shows the size of the operator that we can take in order to ensure local integrability.
\begin{Proposition}\label{general}
Consider, for $\Psi$ positive and increasing, the integro-differential operator of order $\Psi$ defined on a set $E$ (its characteristic function, rather) by
$$
I_E(x)=\int_{\mathbb R^N} \frac{\chi_E(x)-\chi_E(y)}{|x-y|^{N}}\Psi\left(\frac 1{|x-y|}\right)dy,
$$
and let $w_0$ be a modulus of continuity as defined before. Assume that $\Omega$ is a domain whose boundary coincides locally with the graphs of functions $\Phi$ with the property
$|\Phi(x)-\Phi(y)| \le \omega_0(|x-y|)$. Then, if $\Psi$ and $\omega_0$ satisfy the condition
\begin{equation}\label{Dini}
\int_0^1\frac{\omega_0(t)}{t}\Psi\left(\frac 1t\right)dt<\infty,
\end{equation}
we have that $I_{\Omega}(x)$ is locally integrable on $\Omega^c$.
\end{Proposition}

The proof follows the same ideas considered before, particularly those in the geometric proof given in Remark \ref{geometric}, and will be omitted. For the case considered initially, that is when $\omega_0(t)=t^\beta$ and $\Psi(t)=t^\alpha$, we have  $\displaystyle \int_0^1\frac{\omega_0(t)}{t}\Psi\left(\frac 1t\right)dt =\int_0^1\frac{t^{\beta-\alpha}}{t}$dt.  Clearly we need $\alpha<\beta$ for \ref{Dini} to hold. Observe that for a domain resembling the Lebesgue spine, that is, one for which the modulus of continuity $w_0$ satisfies  $w_0^{-1}(r)=e^{-1/r}$ (or, directly,  $\displaystyle w_0(t)=\frac 1{\log \frac 1t} $) there is no $\Psi$ for which condition \ref{Dini} holds.

\

%%%%%%%%%%%%%%%%%%%%%%%%%%%%
%%%%%%%%%%%%%%%%%%%%%%%%%%%%
%%%%%%%%%%%%%%%%%%%%%%%%%%%%

\section{The mixed eigenvalue problem: properties and examples}\label{2}
\setcounter{equation}{0}

Let us consider the sequence of mixed eigenvalue problems
\begin{equation}\label{sequence}
\left\{\begin{array}{rcll}
(-\Delta)^{s} u_1^{k}&=&\lambda_{1}^{k} u_1^{k}&\inn \Omega,\\[1mm]
u_1^{k}&=&0&\inn D_k,\\[1mm]
\mathcal{N}_{s} u_1^k&=&0&\inn N_k,
\end{array}\right.
\end{equation}
where $D_k,\, N_k\subset \Omega^c$ satisfy \eqref{nkdk}
and  $\lambda_1^k = \lambda_1 (D_k) $ and $u_1^k$ represent  the first eigenvalue and the first (normalized in $L^2(\Omega)$)  eigenfunction,  respectively.

\

First, we prove an existence result for the solution of \eqref{sequence}.

\begin{Proposition} \label{u1k}
Given $\Omega$ admissible and  pairs of sets $D_k$ and $N_k$ such that \eqref{nkdk} holds true,
there exists a function $u_1^k \in H_D^s (\Omega)$ satisfying \eqref{sequence}.

Moreover:
\begin{itemize}
\item[i)] $\lambda_1 (\emptyset) = 0 \leq \lambda_1 (D_k)\leq \lambda_1 (\Omega^c)$;
\item[ii)] $u_1^k\geq 0$ in $\mathbb{R}^N$, and  $u_1^k> 0$ in $\Omega\cup N_k $;

\item[iii)] The sequence $\{u_{1}^{k}\}_{k\in\N}$ is uniformly bounded in $L^\infty(\Omega)$;
\item[iv)] $\exists c>0$ such that $\dfrac{1}{|\Omega|} \dint_{\Omega} u_{1}^{k}(y)\,dy\geq c$, uniformly with respect to  $k$.
\end{itemize}
\end{Proposition}

\proof
The existence of the pair $( \lambda_k, u_1^k ) \in \R^+ \times H^s_{D_k} (\Omega) $ follows from \cite{BM}, Proposition 2.4.

Moreover i) is a consequence of the fact that
\begin{equation}\label{R}
\lambda_1^k=\inf_{\mathcal{H}_k} \iint_{Q_\Omega} \dfrac{|u(x)-u(y)|^2}{|x-y|^{N+2s}}\,dxdy,
\end{equation}
where $\mathcal{H}_k := \{ u\in H_{D_k}^{s}(\Omega), \|u\|_{L^2(\Omega)}=1\}$.

We first observe that by Proposition \ref{Poin}  and since
\begin{equation}\label{inclus}
H^s_0 (\Omega) \subset {H^s_D (\Omega)} \subset H^s (\Omega)\,,
\end{equation}
we deduce that  i) holds true.

\vskip 0.5 cm

ii) It immediately follows from the attainability of the first eigenvalue.  For every $k$, there exists $u_1^k\in H_{D_k}^s(\Omega)$, minimum of the Rayleigh quotient in \eqref{R}. But then also $|u_1^k|$ minimizes the quotient, so we can assume $u_1^k\geq 0$.

The fact that $u_1^k>0$ in $\Omega$ follows now as a consequence of the Strong Maximum Principle (see Theorem 1.1 in \cite{BM}).
Moreover, recalling the Neumann boundary condition, we have that
$$
\forall x\in N_k \qquad u (x) \int_{ \Omega}   \dfrac{dy }{|x-y|^{N+2s}}\, dy = \int_{ \Omega}  \dfrac{u(y) \ \ dy }{|x-y|^{N+2s}}\, dy  >0
$$
since $u(y)>0$ in $\Omega$.

\vskip 0.5 cm

iii) We use an argument of Moser type: we set the following convex function
$$
\Phi(\sigma)=\Phi_T(\sigma):=\left\{
\begin{array}{ll}
 \sigma^\beta,  &{\rm if} \quad 0\leq \sigma<T,\\[2mm]
\beta T^{\beta-1} \sigma-(\beta-1)T^\beta, \quad&{\rm if} \quad \sigma \ge T,
\end{array}\right.
$$
for $\beta > 1$ and $T>0$ large. Since
$\Phi$ is Lipschitz (with constant $L_{\Phi}:=\beta T^{\beta-1} $) and $\Phi(0)=0$, then $\Phi(u_1^k) \in H_{D_k}^s(\Omega)$ and, using Proposition $4$ in \cite{LPPS}, we have
\begin{equation}\label{des}
(-\Delta)^{s} \Phi(u_1^k)\le \Phi'(u_{1}^{k}) (-\Delta)^s u_{1}^{k} \qquad \mbox{ in } \Omega.
\end{equation}
Since $\Phi $ is positive, multiplying both sides of \eqref{des} by $\Phi (u_{1}^k)$, and integrating over $\Omega$, we get
$$
\dint_{\Omega } (-\Delta)^s \Phi(u_1^k)(x)\, \Phi (u_1^k)(x)\,dx\leq \lambda_{1}^k \dint_{\Omega } \Phi'(u_1^k)(x)\,\Phi(u_1^k) (x) \,u_1^k (x)\,dx \  .
$$
On the other hand, using the integration by parts formula  \eqref{parts},
$$
\begin{array}{c}
\dint_{\Omega } (-\Delta)^s \Phi(u_1^k) (x)\, \Phi (u_1^k) (x) \,dx \\
\displaystyle = \dfrac{a_{N,s}}{2} \iint_{Q_\Omega} \dfrac{\big(\Phi(u_1^k) (x)-\Phi(u_1^k)(y)\big)^2}{|x-y|^{N+2s}}dxdy - \dint_{\Omega^c } \mathcal{N}_s [\Phi(u_1^k)(x)]\, \Phi (u_1^k) (x)\,dx
\end{array}
$$
and  using the convexity of  $\Phi$, and since $u_1^k (x) \equiv 0$ en $D_k$, then
$$
\begin{array}{c}
\dint_{N_k } \mathcal{N}_s [\Phi(u_1^k)(x)]\, \Phi (u_1^k) (x)\,dx
\leq
- \dint_{N_k } \mathcal{N}_s [u(x)]    \Phi'(u_1^k)(x)\, \Phi (u_1^k) (x)\,dx \leq 0 \,.
\end{array}
$$
Thus, since $\lambda_1^k \leq \lambda_1$ we have
$$
\frac{\mathcal{S}}{2} \|\Phi(u_1^{k})\|^2_{L^{2_s^{*}}(\Omega)}
\leq   \lambda_1 \dint_{\Omega} \Phi'(u_1^k)\,\Phi(u_1^k)\,u_1^k\,dx.
$$
The rest of the proof follows as in   Theorem 13 of \cite{LPPS}.

\vskip 0.5 cm

iv) follows by contradiction, from the fact that $u_1^k\geq 0$ and $\|u_1^k\|_{L^2(\Omega)}=1$, for every $k\in \N$.
\qed

\medskip

Next we study the behavior of the sequence $u_k$ as $k$ diverges.

\medskip

\begin{Proposition}\label{weakly}
Let $\Omega$ be admissible and $N_k$ and $D_k$ as in \eqref{nkdk} and consider the solutions $u^k_1$ of problem \eqref{sequence}. Then there exists a measurable function $u^*$ in $\Omega$,  such that, up to subsequences (not relabeled)
\begin{equation}\label{conv}
u^k_1 \longrightarrow u^*
\qquad
\begin{array}{l}
\mbox{Weakly in }  \ H^s (\Omega), \\
\mbox{Strongly in } \ L^p (\Omega), \ \ \forall p \in[1,2^*_s), \\
\mbox{a.e in }  \ \Omega \,. \\
\end{array}
\end{equation}
\end{Proposition}

\proof Testing a minimizing sequence in \eqref{sequence} with $u_1^{k}$, we find that
\begin{equation}\label{universal}
\dfrac{a_{N,s}}{2} \iint_{Q_\Omega} \dfrac{(u_1^k(x)-u_1^k(y))^2}{|x-y|^{N+2s}}dxdy
=\lambda_1^k\dint_{\Omega}| u_1^k|^2\, dx\le \lambda_{1},
\end{equation}
where $\lambda_1$ is the first eigenvalue of the Dirichlet problem in $\Omega$, i.e.
$$\lambda_1=\lambda_1(\Omega^c)=\inf_{u\in H_{0}^{s}(\Omega), \|u\|_{L^2(\Omega)}=1} \iint_{Q_\Omega} \dfrac{|u(x)-u(y)|^2}{|x-y|^{N+2s}}\,dxdy\,.$$

As a consequence, the  $H^s(\Omega)$- norm of $\{u_1^k\}_{k\in\mathbb{N}}$ is uniformly bounded and then, up to a subsequence, there exists $u^*\in H^s(\Omega)$ such that
$$u_1^k\rightharpoonup u^*,\hbox{   weakly in } H^s(\Omega).$$
By the compact embedding $H^s(\Omega)\subset L^2(\Omega)$ we can assume (again up to a subsequence) that \eqref{conv} holds true.
\qed

\medskip

Let us state an interesting property of the solution of a mixed problem when the set of the prescribed  Neumann condition goes to  infinity.

\begin{Lemma}
 Consider  a function  $u$  such that
$$
\mathcal{N}_s u(x)=0  \qquad \forall x\in N\,,
$$
where $N$ satisfies that
$$
\forall R >0 \qquad N\cap B_R^c \neq \emptyset\,.
$$
Then for all sequences $\{x_j\}_{j}\subset N$
 such that $ |x_j|\to \infty$  as $j \to + \infty$
we have that $\{u(x_j)\}_j $ converges to its average on $\Omega$, that is
$$
\lim_{j \to \infty } u(x_j) = \fint_{\Omega} u(x) dx  .
$$
\end{Lemma}

\proof
Setting  $x_j=|x_j|\theta_j \in N$ with $\theta_j\in \mathbb{S}^{N-1}$, we have
$$
u(x_j)=\dfrac{\dint_{\Omega}  \dfrac{u(y)}{|\theta_j-\frac{y}{|x_j|}|^{N+2s}}\,dy}{\dint_{\Omega} \dfrac{dy}{|\theta_j-\frac{y}{|x_j|}|^{N+2s}}}
\quad \mbox{ with }
\quad
\lim\limits_{|x_j|\rightarrow\infty} \dfrac{1}{|\theta_j-\frac{y}{|x_j|}|^{N+2s}}=1, \quad \forall y\in\Omega.
$$
So, the results follows from the Lebesgue Dominated Convergence Theorem.
\qed

\medskip

It is worth pointing out that the speed of the asymptotics is of order $\dfrac{1}{|x_j|}$. More precisely, $\exists C>0$ such that
$$
\Big|u(x_j)-\fint_{\Omega} u(y)\,dy\Big|\leq \dfrac{C}{|x_j|} \fint_{\Omega} u(y)\,dy,  \qquad \forall j \,.
$$
To prove it, we see that for $x_j=|x_j|\theta_j\in N$ with $\theta_j\in \mathbb{S}^{N-1}$ and
$$g_{x_j}(y):=\dfrac{1}{|\theta_j-\frac{y}{|x_j|}|^{N+2s}},$$
there holds
$$
u(x_j)-\fint_{\Omega} u(y)\,dy
= \dfrac{\dint_{\Omega} u(y) (g_{x_j}(y)-1)\,dy}{\dint_{\Omega} g_{x_j}(y)\,dy} +\fint_{\Omega} u (y)\,dy \left[ \dfrac{\dint_{\Omega} (1-g_{x_j}(y))\,dy}{\dint_{\Omega} g_{x_j}(y)\,dy}\right].
$$
Now, we have
$$
|g_{x_j}(y)-1|= \left| \dfrac{1}{|\theta_j-\frac{y}{|x_j|}|^{N+2s}}-1\right|\leq C \dfrac{|y|}{|x_j|}\leq \dfrac{C'}{|x_j|}, \quad
 \forall y\in\Omega,
$$
and the Lebesgue Dominated Convergence Theorem gives the result again.

\medskip

\medskip

Motivated by the works \cite{CP} and \cite{D} in the local setting, our goal from now on will be to study what happens to the sequence $\{\lambda_1^k\}_{k\in\N}$ when the sets $D_k$ and $N_k$ change with $k$. As we already said, the fact that in the non local framework the {\it boundary} is the whole $\Omega^c$ makes the situation completely different, since the way in which the sets can move or disappear may be much more varied and complicated.

Actually, considering the decay of the kernel of the operator one may think of two ways of {\it diffusing} sets: making their measure tend to zero (which would be the analogous to the local case) or {\it sending them to $\infty$}.

Let us analyze some examples of these situations before giving the rigorous convergence results. Assume the sets $N_k$ to be the ones that we want to dissipate (the examples are analogous for the sets $D_k$).

\subsection{Sets with measure going to zero: shrinking Neumann sets.\\}
\noindent We find here sets satisfying
\begin{enumerate}
\item $N_k\subset   B_R(0) \cap \Omega^c$,
\item $|N_k|\to 0$ as $k\to\infty$,
\end{enumerate}
for some $R>0$, that is, sets contained in a large ball that disappear when $k\rightarrow +\infty$. Notice that this framework includes the case of nested sets, that is,
$$N_1\supseteq N_2\supseteq N_3\supseteq \ldots\supseteq N_k\supseteq \ldots,$$
whose local analogue has been  studied in \cite{CP}, but also the case when $N_k\cap N_{k+1}=\emptyset$, $\forall k \in \N$.

\subsection{Sets disseminating their mass at infinity.}
\begin{itemize}

\item[(i)] Travelling balls: $N_{k}:= B_{r_{k}}(x_k)$, a ball centered at $x_k$, with radius $r_k$, $0< r_k \leq C$ for all $k\in\N$, provided that $|x_k|$ tends to infinity as $k\rightarrow\infty$.

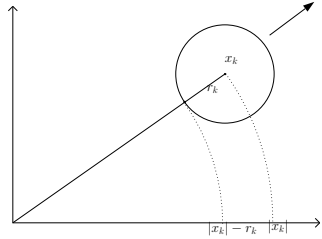
\begin{center}
\psscalebox{0.3 0.3} % Change this value to rescale the drawing.
{
\begin{pspicture}(0,-5.1032505)(13.828284,5.1032505)
\pscircle[linecolor=black, linewidth=0.04, dimen=outer](9.614142,2.0367496){2.2}
\psline[linecolor=black, linewidth=0.04](0.21414214,5.0367494)(0.21414214,-4.5632505)(13.814142,-4.5632505)
\psdots[linecolor=black, dotsize=0.04](9.614142,2.0367496)
\psline[linecolor=black, linewidth=0.04](9.614142,2.0367496)(0.21414214,-4.5632505)
\psline[linecolor=black, linewidth=0.04](0.21414214,5.0367494)(0.41414216,4.8367496)(0.41414216,4.8367496)
\psline[linecolor=black, linewidth=0.04](0.21414214,5.0367494)(0.014142151,4.8367496)
\psline[linecolor=black, linewidth=0.04](13.814142,-4.5632505)(13.614142,-4.3632507)
\psline[linecolor=black, linewidth=0.04](13.814142,-4.5632505)(13.614142,-4.7632504)
\psdots[linecolor=black, dotsize=0.1](9.614142,2.0367496)
\psdots[linecolor=black, dotsize=0.1](7.824142,0.7867495)
\psdots[linecolor=black, dotsize=0.1](7.834142,0.7967495)
\rput[bl](9.604142,2.4867494){\huge $x_k$}
\rput[bl](8.814142,1.1867495){\huge $r_k$}
\psarc[linecolor=black, linewidth=0.04, linestyle=dotted, dotsep=0.10583334cm, dimen=outer](0.21414214,-4.6032505){9.3}{0.0}{35.141083}
\rput[bl](8.854142,-5.1032505){\huge $|x_k|-r_k$}
\psarc[linecolor=black, linewidth=0.04, linestyle=dotted, dotsep=0.10583334cm, dimen=outer](0.38414216,-4.5032506){11.35}{0.0}{35.734894}
\rput[bl](11.514142,-5.0432506){\huge $|x_k|$}
\psline[linecolor=black, linewidth=0.04, arrowsize=0.173cm 4.0,arrowlength=1.5,arrowinset=0.0]{->}(11.594142,3.7567494)(13.574142,5.2367496)
\end{pspicture}
}
\end{center}
\item[(ii)] Travelling rings: $N_{k}:=\{x: R_{k}<|x|<R_{k}+L_{k}\}$, with $0<L_k\leq \infty$ for all $k\in \N$ and $R_k\rightarrow \infty$ as $k\rightarrow\infty$.
\item[(iii)] Travelling strips: $N_k:=\{x=(x^1,x^2,\ldots,x^N): R_k<x^1<R_k+L_k\}$, with $0< L_k\leq \infty$ for all $k\in \N$ and $R_k\rightarrow \infty$ as $k\rightarrow\infty$.
\item[(iv)] {\it Infinite} sectors: $N_{k}:= \{x=|x|\theta;\, |x|>R_{k}, {\theta \in S\subset\mathbb{S}^{N-1}}\}$, provided that $R_k$ tends to infinity as $k\rightarrow\infty$.
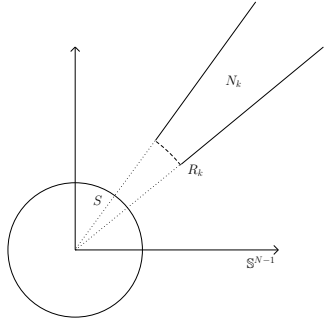
\begin{center}
\psscalebox{0.3 0.3} % Change this value to rescale the drawing.
{
\begin{pspicture}(0,-7.005858)(14.14,7.005858)
\pscircle[linecolor=black, linewidth=0.04, dimen=outer](3.0,-4.005858){3.0}
\psline[linecolor=black, linewidth=0.04](6.56,0.854142)(11.0,6.994142)
\psline[linecolor=black, linewidth=0.04](7.7,-0.20585796)(14.0,4.994142)
\psarc[linecolor=black, linewidth=0.04, linestyle=dashed, dash=0.17638889cm 0.10583334cm, dimen=outer](3.0,-4.005858){6.0}{38.81276}{53.28381}
\psline[linecolor=black, linewidth=0.04, linestyle=dotted, dotsep=0.10583334cm](6.54,0.834142)(2.98,-4.005858)
\psline[linecolor=black, linewidth=0.04, linestyle=dotted, dotsep=0.10583334cm](3.02,-4.005858)(7.7,-0.18585797)
\rput[bl](7.98,-0.74585795){\huge $ R_k$}
\rput[bl](9.66,3.234142){\huge $  N_k$}
\psline[linecolor=black, linewidth=0.04](3.0,-4.005858)(3.0,-4.005858)(3.0,-4.005858)
\psline[linecolor=black, linewidth=0.04](3.0,4.994142)(3.0,-4.005858)
\psline[linecolor=black, linewidth=0.04](3.0,-4.005858)(12.0,-4.005858)
\psline[linecolor=black, linewidth=0.04](12.0,-4.005858)(11.8,-3.805858)(11.8,-3.805858)
\psline[linecolor=black, linewidth=0.04](12.0,-4.005858)(11.8,-4.2058578)
\psline[linecolor=black, linewidth=0.04](3.0,4.994142)(3.2,4.7941422)
\psline[linecolor=black, linewidth=0.04](2.8,4.7941422)(3.0,4.994142)
\rput[bl](10.6,-4.805858){\huge $\mathbb{S}^{N-1}$}
\rput[bl](3.8,-2.005858){\huge $S$}
\end{pspicture}
}
\end{center}
\end{itemize}
Notice that here we can find cases where the measure of the sets $N$ is finite (cases (i) and (ii)) but also, which is more interesting, sets with infinte measure ((iii) and (iv)).
\medskip

We can join all the previous examples (and the corresponding combinations of them) in the following condition:
\begin{equation*}\label{condition}
\hbox{ for every }R>0 \qquad  \lim\limits_{k\rightarrow\infty} |N_{k}\cap B_{R}|=0\,.
\end{equation*}
Indeed, our goal in the next section is to prove that if the above conditions holds, then
$$\lambda_1^k\rightarrow \lambda_1,$$
the first eigenvalue of the Dirichlet problem. Furthermore, we will see that it is a necessary and sufficient condition. Let us point out that, although this is the {\it expectable} result according to \cite{CP}, the non local boundary data makes the conclusion (and the casuistry) not obvious at all.

%%%%%%%%%%%%%%%%%%%%%%%%%%%%%
%%%%%%%%%%%%%%%%%%%%%%%%%%%%%
%%%%%%%%%%%%%%%%%%%%%%%%%%%%%

\section{Dissipating Neumann sets}\label{secNeu}
\setcounter{equation}{0}

Let us consider $\{N_{k}\}_{k\in\N}$ and $\{D_k\}_{k\in\N}$  sequences  of  open sets in $ \Omega^c$ satisfying \eqref{nkdk}. For any $k$, let $\lambda_{1}^{k}$ be the first eigenvalue associated to the mixed problem on $\Omega$  related to  $D_{k}$ and $N_k$ and let  $u_{1}^{k}\in H_{D_{k}}^{s}(\Omega)$, with $\|u_{1}^{k}\|_{L^{2}(\Omega)}=1$, be the corresponding eigenfunction, i.e., the solution to \eqref{sequence}.

Consider now $\varphi_1\in H_{0}^{s}(\Omega)$, the   first positive eigenfunction of the Dirichlet problem, i.e., the solution of
\begin{equation}\label{Dirichlet}
\left\{\begin{array}{ll}
(-\Delta)^{s} \varphi_1 = \lambda_{1}\varphi_1\qquad &\inn \Omega,\\
\varphi_1= 0 &\inn  \Omega^c,
\end{array}\right.
\end{equation}
with $\|\varphi_1\|_{L^{2}(\Omega)}=1$.
We have first the following result.
\begin{Lemma}\label{integrable}
Let $\Omega$ be admissible and define for $x\in\Omega^c$ the function
$$
\Phi(x):= \dint_{\Omega}\dfrac{ \varphi_1(y) }{|x-y|^{N+2s}}\,dy.
$$
Then $\Phi\in L^{1}(\mathbb{R}^{N}\setminus\Omega)$.
\end{Lemma}
\begin{proof} Take $R>0$ such that $\Omega \subset B_R$. Since $\varphi_1$ is bounded, we easily see that $\Phi$ is integrable in $\mathbb R^N\setminus B_{2R}$.

Now, thanks to \cite[Proposition 7.2]{R}, we have that  $\varphi_1\in\mathcal{C}^s(\ren)$ and thus
$$\int_{B_{2R}\setminus\Omega}\dint_{\Omega}\dfrac{ \varphi_1(y) }{|x-y|^{N+2s}}\,dydx=\int_{B_{2R}\setminus\Omega}\dint_{\Omega}\dfrac{ \varphi_1(y)-\varphi_1(x) }{|x-y|^{N+2s}}\,dydx\leq C \int_{B_{2R}\setminus\Omega}\dint_{\Omega}\dfrac{ 1}{|x-y|^{N+s}}\,dydx.$$
Consequently,  by Lemma \ref{punctual},
$$\int_{B_{2R}\setminus\Omega}\dint_{\Omega}\dfrac{ \varphi_1(y) }{|x-y|^{N+2s}}\,dydx\leq C \int_{B_{2R}\setminus\Omega}\frac{1}{dist(x,\partial\Omega)^s}\,dx.$$
We conclude applying Lemma \ref{regularidad}.
\end{proof}

\medskip

We can  prove our first main theorem.
\begin{proof}[Proof of Theorem \ref{theoNeu}] Taking $\varphi_1$ in \eqref{Dirichlet} as a test function in \eqref{sequence} and $u_{1}^{k}$  as a test function in \eqref{Dirichlet}, and integrating by parts as in \eqref{parts}, it follows that
\begin{equation}\label{sub}
\begin{array}{rcl}
(\lambda_{1}-\lambda_{1}^{k}) \dint_{\Omega}\varphi_1(x) u_1^{k}(x)\,dx&=& -\dint_{\mathbb{R}^{N}\setminus\Omega} \mathcal{N}_{s}\varphi_1(x)u_1^{k}(x)\, dx\\
\\
&=& a_{N,s} \dint_{N_{k}} \dint_{\Omega} \dfrac{\varphi_{1}(y) u_{1}^{k}(x) }{|x-y|^{N+2s}}\ dy dx.
\end{array}
\end{equation}
This shows that statement $(A)$ implies the following
$$
(A') \;\;\lim\limits_{k\rightarrow\infty}\dint_{N_{k}} \dint_{\Omega} \dfrac{\varphi_{1}(y) u_{1}^{k}(x) }{|x-y|^{N+2s}}\ dy dx=0.
$$
Conversely, let us prove that $(A')$ implies $(A)$, so that both conditions, $(A)$ and $(A')$, are equivalent. Since $0\le \lambda_{1}^{k}\leq\lambda_{1}$  we have
$$
0\le \liminf_{k\to\infty}  \lambda_1^k \le \limsup_{k\to\infty}  \lambda_1^k \le \lambda_1.
$$
Thus, it is enough to show that $(A')\implies \liminf_{k\to\infty}  \lambda_1^k = \lambda_1$.  So, we take a subsequence converging to the $\liminf$ of the $ \lambda_1^k$'s. We know from \eqref{conv} that we can extract a sub-subsequence
$\{\lambda_1^{k_j}\}_j$ such that
\begin{equation*}\begin{split}
u_{1}^{k_j} &\rightharpoonup u^{*}\inn H^{s}(\Omega) \hbox{ and }
u_{1}^{k_j}\rightarrow u^{*}\inn L^{2}(\Omega)\hbox{ and a.e. in }\Omega,
\end{split}\end{equation*}
with $u^*\gneq 0$. Hence, $$\lim\limits_{j\rightarrow\infty} \dint_{\Omega} \varphi_{1}(x)u_{1}^{k_j}(x)\,dx= \dint_{\Omega} \varphi_{1}(x)u^{*}(x)\,dx >0. $$
Applying this  and the statement $(A')$ in (\ref{sub}) we get that  $\lim_{j\to\infty}  \lambda_1^{k_j} = \lambda_1$, and therefore the initial subsequence converges to $\lambda_1$ too.  Hence, $\liminf_{k\to\infty}  \lambda_1^k = \lambda_1$ and $(A)$ follows.
To finish with the proof of Theorem \ref{theoNeu}, it remains to show that statements $(A')$ and $(B)$ are also equivalent.

\

We first prove that $(B)$ implies $(A')$. To that end, take $R$ large enough so that $\Omega\subset B_{\frac{R}{2}}$,   then
\begin{equation*}\begin{split}
J_{k}&:= \dint_{N_{k}} \dint_{\Omega} \dfrac{\varphi_{1}(y) u_{1}^{k}(x) }{|x-y|^{N+2s}}\ dy dx\\
&\leq \dint_{\mathbb{R}^{N}\setminus B_{R}}\dint_{\Omega} \dfrac{\varphi_{1}(y) u_{1}^{k}(x) }{|x-y|^{N+2s}}\ dy dx+ \dint_{N_k\cap B_{R}}\dint_{\Omega}\dfrac{\varphi_{1}(y) u_{1}^{k}(x) }{|x-y|^{N+2s}}\ dy dx\\
&=: J_{k}^{1}+J_{k}^{2}.
\end{split}\end{equation*}
Using that $\varphi_1$ is bounded, Proposition \ref{u1k} and the fact that if $y\in\Omega$ and $|x|>R$ then $|x-y|>\frac{|x|}{2}$, we have that
$$
J^{1}_{k} \leq C \dint_{\mathbb{R}^{N}\setminus B_{R}}\dint_{\Omega} \Big(\dfrac{1}{|x|/2}\Big)^{N+2s}\ dy dx= C 2^{N+2s} |\Omega|\dint_{|x|>R} \dfrac{dx}{|x|^{N+2s}}= C_{N,s} |\Omega| \dfrac{1}{R^{2s}}.
$$
As a consequence, given $\varepsilon>0$, we can choose $R$ sufficiently large so that $J_{1}^{k}\leq \dfrac{\varepsilon}{2}, \forall k$.

In order to estimate $J_{k}^2$, we make use of Proposition \ref{u1k} and Lemma \ref{integrable},
$$
J_{k}^{2}\leq C\dint_{N_{k}\cap B_{R}}\dint_{\Omega} \dfrac{\varphi_{1} (y)}{|x-y|^{N+2s}}\,dydx= C\dint_{N_{k}\cap B_{R}}\Phi(x)\,dx.
$$
Since the measure $d\mu=\Phi(x)dx$ is absolutely continuous with respect to the Lebesgue measure, there exists $\delta>0$, such that  if $N$ is a measurable subset of $\mathbb{R}^{N}\setminus \Omega,$ with $|N|<\delta$, then
$$
\dint_{N} \Phi(x)\,dx<\frac{\varepsilon}{2C},
$$
for $\varepsilon$ given above. Hence, condition $(B)$ implies the existence of $k_0>0$ such that, $\forall k\geq k_0$, $|N_{k}\cap B_{R}|<\delta$. We conclude that $|J_{k}^{2}|<\frac{\varepsilon}{2}$ and therefore $J_{k}<\varepsilon$, $\forall k\geq k_{0}$.
\medskip

We now prove the other implication, that is, $(A')$ implies $(B)$; assume that for $R$ large we have
\begin{equation}\label{J2}
\lim\limits_{k\rightarrow\infty} J_{k}^{2} = \lim\limits_{k\rightarrow\infty} \dint_{N_{k}\cap B_{R}}\dint_{\Omega} \dfrac{\varphi_{1}(y) u_{1}^{k}(x)}{|x-y|^{N+2s}}\, dydx=0.
\end{equation}
Fix $\delta>0$ small and consider the {\it strips} around $\Omega$,
$$
\Omega_{\delta}:= \{ x \in \Omega^c: \mbox{ dist}(x,\Omega)\leq\delta\}.
$$

\noindent Assume $\Omega_{\delta}\subset B_{R}$. We claim that
\begin{equation}\label{claim}
\lim\limits_{k\rightarrow\infty}|N_{k}\cap B_{R}\setminus\Omega_{\delta}|=0.
\end{equation}
To prove this, we observe that
$$
\begin{array}{rcl}
J_{k}^{2}&\geq& \dint_{N_{k}\cap B_{R}\setminus\Omega_{\delta}}\dint_{\Omega} \dfrac{\varphi_{1}(y) u_{1}^{k}(x)}{|x-y|^{N+2s}}\ dydx\\
&&\\
&\geq&\dfrac{1}{(2R)^{N+2s}} \Big(\dint_{\Omega} \varphi_{1} (y)\,dy\Big)\Big(\dint_{N_{k}\cap B_{R}\setminus\Omega_{\delta}} u_{1}^{k}(x)\,dx\Big),
\end{array}
$$
since $x,y\in B_{R}$ implies $|x-y|\leq 2R$. Thanks to the $\delta$-separation from $\Omega$, we can estimate $u_{1}^{k}(x)$ from below. Indeed, let us recall that on $N_{k}$,
$$
u_{1}^{k}(x)= \dfrac{\dint_{\Omega} \dfrac{u_1^k(y)}{|x-y|^{N+2s}}\, dy}{\dint_{\Omega} \dfrac{dy}{|x-y|^{N+2s}}}.
$$
Now, $y\in\Omega$ and $x\in N_{k}\setminus\Omega_{\delta}$ implies $|x-y|>\delta$ and therefore,
$$
u_{1}^{k}(x)\geq \dfrac{\dfrac{1}{(2R)^{N+2s}}\dint_{\Omega} u_{1}^{k}(y) dy}{\dint_{\Omega} \dfrac{dy}{\delta^{N+2s}}}= \Big(\dfrac{\delta}{2R}\Big)^{N+2s} \dfrac{1}{|\Omega|} \dint_{\Omega} u_{1}^{k} (y)\,dy.
$$
Thanks to Proposition \ref{u1k}, we obtain the estimate
$$
J_{k}^{2}\geq C \Big(\dint_{\Omega} \varphi_{1} (y)\,dy\Big)|N_{k}\cap B_{R}\setminus\Omega_{\delta}|,
$$
where $C$ is a positive constant that depends on $\delta$ and $R$, but independent of $k$. Letting $k\rightarrow \infty$ and applying \eqref{J2} we conclude the proof of the claim.

To obtain $(B)$ we proceed by contradiction. Assume it does not hold. Then, there exists a subsequence, $\{N_{k_{j}}\}_{j\in\N}$ and values $R,\mu>0$ such that
\begin{equation*}\label{cota}
|N_{k_j}\cap B_{R}|\geq \mu, \;\forall j.
\end{equation*}
Since $\lim\limits_{\delta\rightarrow 0^{+}} |\Omega_{\delta}|=0$, we can choose $\delta_{0}>0$ such that $|\Omega_{\delta_0}|\leq\dfrac{\mu}{2}$. But in this case
$$
|N_{k_j}\cap B_{R}\setminus\Omega_{\delta_0}|\geq \dfrac{\mu}{2}, \;\forall j,
$$
which contradicts the claim \eqref{claim}.
\end{proof}
\begin{remark} Let $u^*$ be the function obtained in \eqref{conv}. Observe that we do not know \lq\lq a priori" how $u^*$ is defined  pointwise on $\Omega^c$. However, we have that if $x\in \Omega^c$ and
$ x\notin \limsup N_k=\bigcap_{k=1}^\infty \bigcup_{j=k}^\infty N_j$, then $\displaystyle\exists \lim_{k\to\infty} u_1^k(x)=0$. Hence, if $\left| \limsup N_k\right|=0$, something that can be attained at least for a subsequence,  we can define $u^{**}$   almost everywhere as the pointwise  limit of the corresponding sequence $\{u_1^k\}_k$. In particular,
$u^{**}=u^{*}$ on $\Omega$. Since
$$
\dfrac{a_{N,s}}{2}\dint_{Q_\Omega} \dfrac{|u_{1}^{k}(x)-u_{1}^{k}(y)|^2}{|x-y|^{N+2s}}\,dxdy = \lambda_{1}^{k} \dint_{\Omega} |u_{1}^{k}|^2(x)\,dx\leq \lambda_{1},
$$
by Fatou's Lemma, the fact that $\|u^{**}\|_{L^{2}(\Omega)}=1$ and the uniqueness of the eigenfunction of the Dirichlet problem,  it follows that $u^{**}\in H_0^s(\Omega)$ and that $u^{**}=\varphi_1$.
\end{remark}

%%%%%%%%%%%%%%%%%%%%%%%%%%%%%
%%%%%%%%%%%%%%%%%%%%%%%%%%%%%
%%%%%%%%%%%%%%%%%%%%%%%%%%%%%

\section{Dissipating Dirichlet sets}\label{secDir}
\setcounter{equation}{0}

The aim of this section is to reproduce the analysis performed in Section \ref{secNeu} when the Dirichlet sets dissipate. Indeed, we have the following result.

\begin{Theorem}\label{N1}
Let $\Omega$ be an admissible domain. If $\lim\limits_{k\rightarrow \infty}\lambda_{1}^{k}=0$, then
$$
 \hbox{ for any }R>0\ ,\qquad \lim\limits_{k\rightarrow\infty} |D_{k}\cap B_{R}|=0 \,.
$$
\end{Theorem}

We prove the result adapting the same idea as in Theorem \ref{theoNeu}.

\begin{proof}
Let $\psi_1$ be the first eigenfunction of the Neumann problem in $\Omega$, i.e. it is easy to see that the first eigenvalue is $0$ and thus it satisfies
\begin{equation}\label{probPsi}
\begin{cases}
(-\Delta)^{s} \psi_1 = 0 \qquad &\inn \Omega,\\
\mathcal{N}_{s} \psi_1= 0 &\inn \Omega^c.
\end{cases}
\end{equation}
Notice that, consequently,  $\psi_1= \dfrac{1}{|\Omega|^{1/2}}$: Testing with $u_1^k$ in \eqref{probPsi}, and with $\psi_1$ in \eqref{sequence} we obtain that
\begin{equation}\label{equivPsi}
\lim\limits_{k\rightarrow \infty}\lambda_{1}^{k}=0\qquad \implies \qquad  \lim_{k\rightarrow\infty}\int_{D_k}\int_\Omega\frac{\psi_1(x)u_1^k(y)}{|x-y|^{N+2s}}\,dydx=0.
\end{equation}
The converse statement is also true by an argument similar to the one in the proof of Theorem \ref{theoNeu}.
In particular,
\begin{equation}\label{ik}
\lim_{k\rightarrow\infty}I_k=0,
\end{equation}
where,
$$
 \qquad I_k:=\int_{D_k\cap B_R}\int_\Omega\frac{\psi_1(x)u_1^k(y)}{|x-y|^{N+2s}}\,dydx, \quad  \mbox{ for } \; R \;\mbox{ given}.$$
For $R$ large enough, so that $\Omega \subset B_R$ and using the non negativity of $u_1^k$, we have that \begin{equation*}\begin{split}
I_k&\geq \frac{ 1}{|\Omega|^{1/2}  (2R)^{N+2s}}\left |D_k\cap B_R\right |\int_\Omega u_1^k\,dy\,.
\end{split}\end{equation*}
Now, from Proposition \ref{u1k} we have
\begin{equation*}
I_k \geq \frac{c\, |\Omega|^{1/2} }{(2R)^{N+2s}}\left |D_k\cap B_R\right |,
\end{equation*}
uniformly in $k$. Therefore, if \eqref{ik} holds, we conclude that
$$\lim_{k\rightarrow\infty}|D_k\cap B_R|=0.$$
\end{proof}

If we want to use the techniques developed in Theorem \ref{theoNeu} to prove the other implication, the lack of regularity of the functions $u_1^k$ reduces the problem to the integrability of the kernel of the operator near  $\partial \Omega$.
This yields to a restriction on the range of  admissible $s$.

\begin{Theorem}\label{N2}
Consider $\Omega$ an admissible  domain and  $0<s<1/2$. If
$$
 \hbox{ for all }R>0, \qquad \lim\limits_{k\rightarrow\infty} |D_{k}\cap B_{R}|=0,
$$
then $\lim\limits_{k\rightarrow \infty}\lambda_{1}^{k}=0$.
\end{Theorem}

\begin{proof}
Thanks to the converse statement in \eqref{equivPsi}, our goal is to prove that
$$\lim_{k\rightarrow\infty}\int_{D_k}\int_\Omega\frac{\psi_1(x)u_1^k(y)}{|x-y|^{N+2s}}\,dydx=0,$$
where $\psi_1= \frac{1}{|\Omega|^{1/2}}$ solves \eqref{probPsi}. Suppose $R$ large enough so that $\Omega\subset B_{R/2}$. Notice first that, by Proposition \ref{u1k}, there exists $C=C(N,\Omega,s)$ such that
\begin{equation*}
I_k^1:=\int_{\ren\setminus B_R}\int_\Omega\frac{\psi_1(x)u_1^k(y)}{|x-y|^{N+2s}}\,dydx\leq C \int_{\ren\setminus B_R}\int_\Omega\left(\frac{1}{|x|/2}\right)^{N+2s}dydx\leq \frac{\tilde{C}}{R^{2s}}.
\end{equation*}
Thus, given any $\varepsilon>0$, we can choose $R$ large enough such that
\begin{equation}\label{I1}
I_k^1\leq \frac{\varepsilon}{2}.
\end{equation}
Furthermore, by Proposition \ref{u1k} iii), we have that
\begin{equation}\label{I2}
I_k^2:=\int_{D_k\cap B_R}\int_\Omega\frac{\psi_1(x)u_1^k(y)}{|x-y|^{N+2s}}\,dydx\leq C\int_{D_k\cap B_R}\int_\Omega\frac{1}{|x-y|^{N+2s}}\,dydx.
\end{equation}
Since $0<2s<1$, we can apply Lemma \ref{punctual} and Lemma \ref{regularidad} to conlude that $I_k^2<\frac{\varepsilon}{2}$ for some $k$ large enough. Therefore, for every $\varepsilon >0$ there exists $k_0>0$ such that
$$\int_{D_k}\int_\Omega\frac{\psi_1(x)u_1^k(y)}{|x-y|^{N+2s}}\,dydx\leq I_k^1+I_k^2<\varepsilon\;\;\hbox{ for every }\;\;k\geq k_0,$$
and the result follows.
\end{proof}

\begin{remark}
Notice that the restriction on $s$ arises in order to estimate the term in \eqref{I2}. Here, to apply the integrability Lemma \ref{regularidad} we need $2s$ to be less than 1. This restriction  does not appear in the case treated in Section \ref{secNeu} because we can take advantage of the regularity of $\varphi_1$, the eigenfunction of the Dirichlet problem, to reduce the singularity of the kernel.
\end{remark}

In the case $1/2\leq s <1$ we can give partial results. In particular, we can prove the result when the Dirichlet sets do not collapse to the boundary of $\Omega$.

\begin{Proposition}\label{DirFar}
Let $s\in (0,1)$. If
$$
\lim\limits_{k\rightarrow\infty} |D_{k}\cap B_{R}|=0 \qquad \hbox{ for all }R>0,
$$
and
$$\exists\,\delta, k_0>0 \;\;\hbox{ s.t. }\;\;dist(D_k,\Omega)>\delta\;\;\;\forall \,k\geq k_0,$$
then $\lim\limits_{k\rightarrow \infty}\lambda_{1}^{k}=0$ up to a subsequence.
\end{Proposition}

\begin{proof}
The result follows just by noticing that, for $k$ large enough, \eqref{I2} can be replaced by
$$I_k^2\leq C\delta^{-(N+2s)} |D_k\cap B_R|,$$
since $|x-y|\geq \delta$ whenever $x\in D_k\cap B_R$ and $y\in\Omega$.
\end{proof}

Finally, to study the case of Dirichlet sets arbitrarily close to $\Omega$, we introduce the following condition:
\begin{equation*}
(C) \qquad \qquad \lim_{k\to\infty}\int_{D_k}\int_\Omega \frac 1{|x-y|^{N+2s}}dydx=0.
\end{equation*}

\begin{Proposition}\label{condC}
Fix $s\in (0,1)$. Let $\Omega,  D_k, N_k$ be as in (\ref{nkdk}) and $\{\lambda^k_1\}_k$, $\{u^k_1\}_k$ the corresponding sequences of eigenvalues and eigenfunctions as before. Then, if $(C)$ holds for  $s$, we have
$$\lim\limits_{k\rightarrow \infty}\lambda_{1}^{k}=0.$$
\end{Proposition}

\begin{remark}
As one can easily see, condition $(C)$ is slightly stronger than the condition $(B)$ defined in the statement of Theorem \ref{theoNeu}. This is because, assuming $\Omega \subset B_R$, we have
$$
\int_{D_k}\int_\Omega \frac 1{|x-y|^{N+2s}}dydx\ge \int_{D_k\cap B_R}\int_\Omega \frac 1{|x-y|^{N+2s}}dydx\ge
\frac{|\Omega| }{(2R)^{N+2s}}\left |D_k\cap B_R\right |.
$$
On the other hand, observe that in the situation of Theorem \ref{N2} ($0<s<1/2$) and Proposition \ref{DirFar} ($D_k$ \lq\lq away" from
$\Omega$), $(B)$ and $(C)$ are in fact equivalent. 
\end{remark}

\vskip 3mm

\noindent \textit{Proof of the Proposition:} Since the  $\lambda^k_1$'s are positive and bounded, we only need to prove that
\begin{equation*}
\limsup_{k\to\infty} \lambda^k_1=0.
\end{equation*}
We take a subsequence $\{\lambda^{k_j}_1\}_j$ converging to the value $\lambda^*=\limsup_{k\to\infty} \lambda^k_1$. By taking a sub-subsequence if needed, we can assume that the corresponding sequence of eigenfunctions
$\{u^{k_j}_1\}_j$ converges weakly in $H^s(\Omega)$ to the function $u^*$ obtained in Proposition \ref{weakly}. Now, given a (bounded) test function $\varphi$ we have
\begin{equation}\label{C1}
\frac{a_{N,s}}{2}\iint_{Q_\Omega}\frac{(u^{k_j}_1(x)-u^{k_j}_1(y))(\varphi(x)-\varphi(y))}{|x-y|^{N+2s}}\,dx dy
= \lambda^{k_j}_1\int_\Omega \varphi \,u^{k_j}_1\,dx -\int_{D_{k_j}} \varphi \,\mathcal{N}_su^{k_j}_1\, dx.
\end{equation}
Using condition $(C)$, we have
$$
\lim_{j\to\infty} \left| \int_{D_{k_j}} \varphi \,\mathcal{N}_s u^{k_j}_1\, dx\right|
\le \lim_{j\to\infty}\int_{D_{k_j}}\int_\Omega \frac {|\varphi(x)|\, u^{k_j}_1(y)}{|x-y|^{N+2s}}dydx
\le C  \lim_{j\to\infty}\int_{D_{k_j}}\int_\Omega \frac 1{|x-y|^{N+2s}}dydx=0.
$$
So, taking limits in (\ref{C1}) we obtain
\begin{equation}
\frac{a_{N,s}}{2}\iint_{Q_\Omega}\frac{(u^*(x)-u^*(y))(\varphi(x)-\varphi(y))}{|x-y|^{N+2s}}\,dx dy
= \lambda^*\int_\Omega \varphi \,u^*\,dx
\end{equation}
Since we also have
\begin{equation}
\mathcal{N}_su^*(x)=0, \quad \mbox{a.e. on } \Omega^c,
\end{equation} we deduce that $u^*$ is a solution to the problem
\begin{equation*}
\left\{\begin{array}{rcll}
(-\Delta)^{s} v&=&\lambda^* v&\inn \Omega,\\
\mathcal{N}_{s} v&=&0&\inn \Omega^c.
\end{array}\right.
\end{equation*}
Hence, either $u^*\equiv 0$, a contradiction with the fact that $\dint_\Omega \left(u^*\right)^2 dx=1$, or $\lambda^*=0$ as we wanted.

\qed

\medskip

As an example of a configuration of $D_k$'s, $N_k$'s for which Proposition \ref{condC} applies, we have the following

\begin{Corollary}\label{coro}
Let $\Omega$ be convex. Consider a collection of open balls $\{D_k\}$ arbitrarily distributed on $\Omega^c$ and assume that $\lim_{k\to\infty} |D_k|=0$. Then, if $N\ge 3$ and $0<s<1$ of if $N=2$ and $0<s<3/4$ we have for the corresponding sequence of eigenvalues $\{\lambda^{k}_1\}_j$ that
$$
\lim_{k\to\infty} \lambda^{k}_1=0.
$$
\end{Corollary}

\begin{proof} From Proposition \ref{condC} we only need to show that $(C)$ holds for $\Omega$ and
$\{D_k\}$ in the range of $s$ and $N$ defined in the hypothesis; that is,
$$
\lim_{k\to\infty}\int_{D_k}\int_\Omega \frac 1{|x-y|^{N+2s}}dydx=0, \quad \mbox{for }
N\ge 3, 0<s<1, \mbox{or } N=2, 0<s<3/4.
$$
As in Section \ref{1}, we have the estimate
$$
\int_\Omega \frac 1{|x-y|^{N+2s}}dy \le \frac {C_N}{dist(x,\partial \Omega)^{2s}}.
$$
Now, since $\Omega$ and $D_k$ are disjoint open convex sets, they are separated by a hyperplane. By rotation invariance, we may assume with no loss of generality that this hyperplane is $x_N=0$ and that $D_k$ is simply
the ball of radius $r$ centered at $(0,\dots,0,r)$, which we will denote by $B_r^0$. Define now
$$
E(r)=\int_{B_r^0}  \frac {1}{dist(x,\partial \Omega)^{2s}} dx.
$$
Clearly, if $x=(x',x_N)\in B_r^0$ we have $dist(x,\partial \Omega) \ge x_N$. Therefore,
$$
E(r)\le \int_{B_r^0}  \frac {1}{x_N^{2s}} dx'dx_N=
\int_0^{2r}  \frac {1}{x_N^{2s}} \left|\{x'\in\mathbb R^{N-1}: (x',x_N)\in B_r^0\} \right|dx_N.
$$
Observe that if $x=(x',x_N)\in B_r^0$ then
$$
|x'|^2+(r-x_N)^2<r^2, \quad \mbox{i.e., } |x'|<\sqrt{2rx_N-x_N^2}.
$$
So, $\{x'\in\mathbb R^{N-1}: (x',x_N)\in B_r^0\}$ represents a ball in $\mathbb R^{N-1}$ of radius
$$
\sqrt{2rx_N-x_N^2}<(2rx_N)^{\frac 12}.
$$
Hence,
$$
E(r)\le C_N \int_0^{2r}  \frac {1}{x_N^{2s}} (2rx_N)^{\frac {N-1}2} dx_N=C_N' r^{N-2s},
$$
provided that $\displaystyle \frac{N-1}2-2s>-1$, or equivalently, $\displaystyle s<\frac{N+1}4$. If this is the case, then
$$
E(r)=o(1), \quad \mbox{as } r\longrightarrow 0,
$$
and the corollary follows.

\end{proof}

\medskip

\begin{remark}
The results contained in this section show the fundamental differences with the local case. In particular, Proposition \ref{DirFar} allows us to conclude that the local result by Denzler (Theorem \ref{denzlerInf}) does not hold in the non local case. Indeed, given a fixed value $\alpha$ (even infinity) we can find a configuration of domains whose measures tend to $\alpha$ such that the corresponding eigenvalues get as small as we want. This is done simply by {\it sending the Dirichlet sets to $\infty$} .
\end{remark}

\medskip

\begin{remark}
In Theorem \ref{denzlerSup}, J. Denzler shows also that for any given value, one can choose a configuration of  sets with Dirichlet condition so that one recovers the eigenvalue of the whole Dirichlet problem. That is, placing cleverly the Dirichlet sets along the boundary of $\Omega$, no matter how small they are, the eigenvalue of the mixed problem behaves like the Dirichlet one.

According to Proposition \ref{DirFar}, if this happens in the non local case it would have to be when the Dirichlet part touches the boundary of $\Omega$.
Corollary \ref{coro} shows an example where the sets can be placed touching $\partial \Omega$ but so that  the above phenomenon does not hold. Whether or not  this is true for any family of sets remains  an open question.
\end{remark}

%%%%%%%%%%%%%%%%%%%%%%%%%%%%%%%
%%%%%%%%%%%%%%%%%%%%%%%%%%%%%%%
%%%%%%%%%%%%%%%%%%%%%%%%%%%%%%%

\end{document}